\documentclass[12pt]{article}%
\usepackage[intlimits]{amsmath}
\usepackage[english]{babel}
\usepackage{amssymb}
\usepackage{latexsym}
\usepackage{tikz}
\usepackage{anysize}
\usepackage{subfigure}
\usepackage[T1]{fontenc}
\usepackage[sc]{mathpazo}
\usepackage{color}
\usepackage[colorlinks]{hyperref}
\usepackage{amsfonts}
\usepackage{graphicx}%
\definecolor {refcol}{RGB}{40,0,255}
\hypersetup{colorlinks=true,allcolors=refcol}
\makeatletter
\newfont{\footsc}{cmcsc10 at 8truept}
\newfont{\footbf}{cmbx10 at 8truept}
\newfont{\footrm}{cmr10 at 10truept}
\newtheorem{theorem}{Theorem}[section]

\newtheorem{corollary}{Corollary}[section]

\newtheorem{definition}{Definition}[section]
\newtheorem{example}{Example}

\newtheorem{lemma}{Lemma}

\newtheorem{problem}[theorem]{Problem}
\newtheorem{proposition}[theorem]{Proposition}

\newtheorem{remark}{Remark}[section]

\newenvironment{proof}[1][Proof]{\noindent{\textbf {#1}  }}  {\hfill$\Box$\bigskip}

\begin{document}
\title{\textbf{New results about the Inverse Eigenvalue Problem of a Graph}}
\author{Roberto C. D\'iaz$^{a}$\thanks{
Corresponding author} \ and \ Ana I. Julio$^{a}$
\\
$^{a}${\small Departamento de Matem\'{a}ticas, Universidad Cat\'{o}lica del
Norte }\ \\
{\small Casilla 1280, Antofagasta, Chile.}}
\date{}
\maketitle
\title{}

\begin{abstract}
All graphs considered are simple and undirected. The Inverse Eigenvalue Problem of a Graph $G$ (IEP-G) aims to find all possible
spectra for matrices whose $(i,j)-$entry, for $i\neq j$, is nonzero precisely when $i$ is adjacent to $j$. A cluster in a graph $G$ is a pair of vertex subsets $(C, S)$, where $C$ is a maximal set of cardinality $\vert C\vert\geq 2$ of independent vertices sharing the same set $S$ of $\vert S\vert$ neighbors. Let $G$ be a connected graph on $n$ vertices with a cluster $(C, S)$ and $H$ be a graph of order $\vert C\vert$. Let $G(H)$ be the connected graph obtained from $G$ and $H$ when the edges of $H$ are added to the edges of $G$ by identifying the vertices of $H$ with the vertices in $C$. In this paper, we construct a symmetric matrix with associated complete graph, which satisfies some interesting properties. This result is applied to obtain new sufficient conditions on the IEP-G, when $G$ is a graph of order $n$ having a clique of order $k$ and a cluster $(C,S)$, where $\vert C\vert=n-k$ and $\vert S\vert=r\leq k$, as well as, for the graph $G(K_{n-k})$, being $G$ as before. In particular, when $G$ is a graph obtained from $K_{n}$ by deleting a single edge $e_{k}\in E(K_{n})$, we establish a necessary and sufficient condition on the IEP-G. Several illustrative example are given. The constructive nature of our results generate algorithmic procedures that always allow one to compute a solution matrix.
\end{abstract}


\section{Introduction}
Let $G=\left( V\left( G\right) ,E\left( G\right) \right)$ be a simple undirected graph with vertex set $V(G)=\{1,\ldots,n\}$ and edge set $E(G)=\{ij \ : \ i, j\in V(G)\}$. The cardinality of $V\left( G\right)$ is called the order of $G$. When
$e_{k}=ij\in E\left( G\right)$ has end vertices $i$ and $j$, then we say that $i$ and $j$ are adjacent. If
$i\in V\left( G\right)$, then $N_{G}(i)$ is the set of neighbors of $i$ in $G$, that is,
$N_{G}(i) = \{j\in V\left( G\right) : ij\in E\left( G\right)\}$. The cardinality of $N_{G}(i)$ is said to be the \textit{degree} of $i$.
\\

The adjacency matrix of a graph $G$ of order $n$, $A(G)$, is a $0-1$-matrix of order $n$ with entries $a_{ij}$ such that $a_{ij}=1$
if $ij\in E(G)$ and $a_{ij}=0$ otherwise. Other matrices on $G$ are the Laplacian matrix $L\left( G\right) =D\left( G\right)-A\left( G\right)$
and the signless Laplacian matrix $Q\left( G\right)=D\left( G\right) +L\left( G\right) ,$ where $D\left( G\right) $ is the
diagonal matrix of vertex degrees. It is well known that $L\left( G\right) $ and $Q\left( G\right) $ are positive semidefinite matrices
and that $\left(0,\mathbf{1}\right) $ is an eigenpair of $L\left( G\right) $ where $\mathbf{1}$ is the all ones vector. Other matrices on $G$ are the {\it normalized adjacency matrix} $\widehat{\mathcal{A}}(G)=\sqrt{D(G)}^{-1}A(G)\sqrt{D(G)}^{-1}$, the {\it normalized Laplacian matrix} $\widehat{\mathcal{L}}(G)=\sqrt{D(G)}^{-1}L(G)\sqrt{D(G)}^{-1}=I_{n}-\widehat{\mathcal{A}}(G)$ and the {\it normalized signless Laplacian matrix} $\widehat{\mathcal{Q}}(G)=\sqrt{D(G)}^{-1}Q(G)\sqrt{D(G)}^{-1}=I_{n}+\widehat{\mathcal{A}}(G)$, where $\sqrt{D(G)}=diag\{\sqrt{d_{1}},\ldots,\sqrt{d_{n}}\}$.
\\

By $S_{n}(\mathbb{R})$ we will denote the set of real symmetric $n\times n$ matrices. Given $M\in S_{n}(\mathbb{R})$, $M(i)$ is the submatrix of $M$ with the $i-$th row and $i-$th column removed, $\lambda_{i}(M)$, $i=1,\ldots,n$, its eigenvalues and the spectrum (the multiset of eigenvalues) of $M$ will be denoted by $\sigma(M)=\{\lambda_{1}(M),\lambda_{2}(M),\ldots,\lambda_{n}(M)\}$. By $\lambda^{[m]}$ we will denote the eigenvalue $\lambda$ with algebraic multiplicity $m$.
\\

This paper is devoted to study the {\it Inverse Eigenvalue Problem of a Graph} (IEP-G). Before establishing information about this problem, we need an additional definition: given $A=(a_{ij})\in S_{n}(\mathbb{R})$, the graph of $A$, $\mathcal{G}(A)$, is the graph with vertex set $\{1,\ldots,n\}$ and edges set $\{ij \ : \ a_{ij}\neq 0 \ \ \text{and} \ \ i\neq j\}$. Note that there is no restriction on the diagonal entries of $A$ for determining $\mathcal{G}(A)$. For example, for the matrix
\begin{align*}
\definecolor{qqqqff}{rgb}{0.,0.,1.}
\begin{tikzpicture}
\draw[line width=0.8pt,] (0.,1.5)-- (1.5,1.5);
\draw[line width=0.8pt,] (1.5,1.5)-- (0.,0.);
\draw[line width=0.8pt,] (1.5,1.5)-- (3.,0.);
\draw[line width=0.8pt,] (-1.5,0.)-- (0.,0.);
\draw[line width=0.8pt,] (0.,0.)-- (3.,0.);
\begin{scriptsize}
\draw [fill=black] (0.,0.) circle (1.5pt);
\draw[color=black] (0.,-0.3) node {$2$};
\draw [fill=black] (-1.5,0.) circle (1.5pt);
\draw[color=black] (-1.5,-0.3) node {$1$};
\draw [fill=black] (3.,0.) circle (1.5pt);
\draw[color=black] (3.,-0.3) node {$3$};
\draw [fill=black] (1.5,1.5) circle (1.5pt);
\draw[color=black] (1.5,1.8) node {$4$};
\draw [fill=black] (0.,1.5) circle (1.5pt);
\draw[color=black] (0.,1.8) node {$5$};
\draw (1.5,-0.8) node {$\mathcal{G}(A)$};
\draw (-3.,0.75) node {$\rightsquigarrow$};
\draw (-7.,0.7) node {{\normalsize $A=\begin{bmatrix}
-1 & \frac{1}{2} & 0 & 0 & 0 \\
\frac{1}{2} & 0 & -2 & 1 & 0 \\
0 & -2 & 1 & -1 & 0 \\
0 & 1 & -1 & 1 & -\frac{3}{2} \\
0 & 0 & 0 & -\frac{3}{2} & 2
\end{bmatrix}$}};
\end{scriptsize}
\end{tikzpicture}
\end{align*}

Given a simple undirected graph $G=(V(G),E(G))$, we define by
\[\mathcal{S}(G)=\{A\in S_{n}(\mathbb{R}) : \mathcal{G}(A)=G\}.\]
Note that $A_{G}, \ L_{G}, \ Q_{G}, \ \widehat{\mathcal{A}}_{G}, \ \widehat{\mathcal{L}}_{G}, \ \widehat{\mathcal{Q}}_{G} \ \in \ \mathcal{S}(G)$. The problem of characterizing all lists of real numbers $\{\lambda_{1}, \lambda_{2},\cdots,\lambda_{n}\}$ that can be the spectrum of a matrix $A\in \mathcal{S}(G)$ is known as the {\it Inverse Eigenvalue Problem of a Graph $G$} (IEP-G). This problem is hard to answer in general and it is motivated from the theory of vibrations \cite{Gant, Grah}. The study of the vibrations of a string leads to classical {\it inverse Sturm-Liouville problems}. When $G$ is a path, the IEP-G corresponds to a discretization of the inverse Sturm-Liouville Problem for the string, and leads to the classical study of the inverse eigenvalue problem for Jacobi matrices that was solved by the sequence of papers \cite{Gray, Hald, Hoch}. For this reason, the IEP-G can be viewed as a discrete version of the question: {\it What kind of vibration behaviors (spectra) are allowed on a given structure (graph)}?. \\

As usual, let $P_{n},C_{n},S_{n}$ and $K_{n}$ be the path, cycle, star and the complete graph on $n$ vertices, respectively. Next, we establish a couple of definitions that will play an important role in this study:

\begin{definition}\label{Def1}
{\rm \cite{Car, Di2, Merr} A {\it cluster} in a graph $G$ is a pair of vertex subsets $(C, S)$, where $C$ is a maximal set of independent vertices (vertices not adjacent to each other) of cardinality $\vert C \vert\geq 2$ sharing the same set $S$ of $\vert S \vert$ neighbors.}
\end{definition}

\begin{example}\label{Ex1} Let $G$ be the graph depicted in Figure 1:
\begin{equation*}
\begin{tikzpicture}
\draw [line width=0.7pt] (-2.,0.)-- (-1.,-1.5);
\draw [line width=0.7pt] (-1.,-1.5)-- (1.,-1.5);
\draw [line width=0.7pt] (1.,-1.5)-- (2.,0.);
\draw [line width=0.7pt] (2.,0.)-- (1.,1.5);
\draw [line width=0.7pt] (1.,1.5)-- (-1.,1.5);
\draw [line width=0.7pt] (-1.,1.5)-- (-2.,0.);
\draw [line width=0.7pt] (-1.,1.5)-- (1.,-1.5);
\draw [line width=0.7pt] (1.,1.5)-- (-1.,-1.5);
\draw [line width=0.7pt] (-2.,0.)-- (2.,0.);
\draw [line width=0.7pt] (-1.,1.5)-- (-1.,-1.5);
\draw [line width=0.7pt] (1.,1.5)-- (1.,-1.5);
\draw [line width=0.7pt] (-2.,0.)-- (1.,1.5);
\draw [line width=0.7pt] (-2.,0.)-- (1.,-1.5);
\draw [line width=0.7pt] (2.,0.)-- (-1.,-1.5);
\draw [line width=0.7pt] (2.,0.)-- (-1.,1.5);
\draw (2.,4.)-- (1.,1.5);
\draw (0.,3.)-- (1.,1.5);
\draw (-2.,4.)-- (1.,1.5);
\draw (-1.,1.5)-- (2.,4.);
\draw (-1.,1.5)-- (0.,3.);
\draw (-2.,4.)-- (-1.,1.5);
\begin{scriptsize}
\draw [fill=black] (-2.,0.) circle (2.0pt);
\draw (-2.3,0.) node {$9$};
\draw [fill=black] (2.,0.) circle (2.0pt);
\draw (2.3,0.) node {$6$};
\draw [fill=black] (-1.,-1.5) circle (2.0pt);
\draw (-1.,-1.8) node {$8$};
\draw [fill=black] (1.,-1.5) circle (2.0pt);
\draw (1.,-1.8) node {$7$};
\draw [fill=black] (-1.,1.5) circle (2.0pt);
\draw (-1.3,1.5) node {$4$};
\draw [fill=black] (1.,1.5) circle (2.0pt);
\draw (1.3,1.5) node {$5$};
\draw [fill=black] (0.,3.) circle (2.0pt);
\draw (0.,3.3) node {$2$};
\draw [fill=black] (-2.,4.) circle (2.0pt);
\draw (-2.,4.3) node {$1$};
\draw [fill=black] (2.,4.) circle (2.0pt);
\draw (2.,4.3) node {$3$};
\draw (0.,-2.5) node {\text{Fig 1. Graph} $G$};
\end{scriptsize}
\end{tikzpicture}
\end{equation*}
We see that $(C, S)$, where $C=\{1,2,3\}$ and $S=\{4,5\}$ is a cluster in $G$.
\end{example}

\begin{definition}
Let $G$ be a graph having a cluster $(C, S)$. Let $H$ be a graph of order $\vert C \vert$. Let $G(H)$
be the graph obtained from $G$ and $H$ when the edges of $H$ are added to the edges of $G$ by identifying the
vertices of $H$ with the vertices in $C$.
\end{definition}

\begin{example}\label{Ex2} Let $G$ be the graph of Example \ref{Ex1}. The graph $G(K_{3})$ is depicted in Figure 2:
\begin{equation*}
\begin{tikzpicture}
\draw [line width=0.7pt] (-2.,0.)-- (-1.,-1.5);
\draw [line width=0.7pt] (-1.,-1.5)-- (1.,-1.5);
\draw [line width=0.7pt] (1.,-1.5)-- (2.,0.);
\draw [line width=0.7pt] (2.,0.)-- (1.,1.5);
\draw [line width=0.7pt] (1.,1.5)-- (-1.,1.5);
\draw [line width=0.7pt] (-1.,1.5)-- (-2.,0.);
\draw [line width=0.7pt] (-1.,1.5)-- (1.,-1.5);
\draw [line width=0.7pt] (1.,1.5)-- (-1.,-1.5);
\draw [line width=0.7pt] (-2.,0.)-- (2.,0.);
\draw [line width=0.7pt] (-1.,1.5)-- (-1.,-1.5);
\draw [line width=0.7pt] (1.,1.5)-- (1.,-1.5);
\draw [line width=0.7pt] (-2.,0.)-- (1.,1.5);
\draw [line width=0.7pt] (-2.,0.)-- (1.,-1.5);
\draw [line width=0.7pt] (2.,0.)-- (-1.,-1.5);
\draw [line width=0.7pt] (2.,0.)-- (-1.,1.5);
\draw [line width=0.7pt] (0.,3.)-- (2.,4.);
\draw [line width=0.7pt] (-2.,4.)-- (2.,4.);
\draw [line width=0.7pt] (-2.,4.)-- (0.,3.);
\draw (2.,4.)-- (1.,1.5);
\draw (0.,3.)-- (1.,1.5);
\draw (-2.,4.)-- (1.,1.5);
\draw (-1.,1.5)-- (2.,4.);
\draw (-1.,1.5)-- (0.,3.);
\draw (-2.,4.)-- (-1.,1.5);
\begin{scriptsize}
\draw [fill=black] (-2.,0.) circle (2.0pt);
\draw (-2.3,0.) node {$9$};
\draw [fill=black] (2.,0.) circle (2.0pt);
\draw (2.3,0.) node {$6$};
\draw [fill=black] (-1.,-1.5) circle (2.0pt);
\draw (-1.,-1.8) node {$8$};
\draw [fill=black] (1.,-1.5) circle (2.0pt);
\draw (1.,-1.8) node {$7$};
\draw [fill=black] (-1.,1.5) circle (2.0pt);
\draw (-1.3,1.5) node {$4$};
\draw [fill=black] (1.,1.5) circle (2.0pt);
\draw (1.3,1.5) node {$5$};
\draw [fill=black] (0.,3.) circle (2.0pt);
\draw (0.,3.3) node {$2$};
\draw [fill=black] (-2.,4.) circle (2.0pt);
\draw (-2.,4.3) node {$1$};
\draw [fill=black] (2.,4.) circle (2.0pt);
\draw (2.,4.3) node {$3$};
\draw (0.,-2.5) node {\text{Fig 2. Graph} $G(K_{3})$};
\end{scriptsize}
\end{tikzpicture}
\end{equation*}
\end{example}

The purpose of the paper is to advance towards the knowledge of a solution of the IEP-G. Our paper is organized as follows: In Section 2, some background and preliminaries of the main concepts used through the text are introduced. In Section 3, we establish our main results. In particular, we construct a matrix $A\in \mathcal{S}(K_{n})$ with prescribed spectrum $\Lambda =\{\lambda_{1},\lambda _{2},\ldots ,\lambda _{n}\}$ where $\lambda _{1}\neq \lambda_{2}$, which satisfies some interesting properties. This result is applied to obtain new sufficient conditions on the IEP-G, when $G$ is a graph of order $n$ having a clique of order $k$ and a cluster $(C,S)$, where $\vert C\vert=n-k$ and $\vert S\vert=r\leq k$, as well as, for $G(K_{n-k})$, being $G$ as before. In particular, when $G$ is a graph obtained from $K_{n}$ by deleting a single edge $e_{k}\in E(K_{n})$, we obtain a necessary and sufficient condition. Some illustrative example are given. The constructive nature of our results generate algorithmic procedures that always allow one to compute a solution matrix.

\section{Preliminaries}
The IEP-G has been solved for only a handful of families of graphs. In this section, we provide some known background about the IEP-G, as well as, we introduce the preliminaries of the main concepts used through the text.
\\ \\
$\blacktriangleright$ {\bf Paths $P_{n}$ and Cycles $C_{n}$:} the IEP-G, when $G=P_{n}$, has been solved when studying the inverse eigenvalue problem for Jacobi matrices
\begin{align}\label{MJac}
\begin{bmatrix}
a_{1} & b_{1} & & & & b_{n} \\
b_{1} & a_{2} & b_{2} \\
& b_{2} & a_{3} & b_{3} \\
& & \ddots & \ddots & \ddots \\
& & & & a_{n-1} & b_{n-1} \\
b_{n} & & & & b_{n-1} & a_{n} \\
\end{bmatrix}, \ \ b_{k}>0 \ \ \text{for all} \ \ k=1,\ldots,n-1, \ \ \text{and} \ \ b_{n}=0,
\end{align}
which was studied in \cite{Hald}, as well as in \cite{Gray, Hoch}, where the authors presented a constructive method for calculating the entries of the
matrix explicitly in terms of the given eigenvalues; while the IEP-G, when $G=C_{n}$, has been solved when studying the inverse eigenvalue problem for periodic Jacobi matrices (i.e., in \eqref{MJac} $b_{k}>0$, for all $k=1,\ldots,n$). Ferguson in \cite{Ferg} studies spectral properties of Jacobi and periodic Jacobi matrices. The author presents an algorithms for the construction of Jacobi and periodic Jacobi matrices with prescribed spectra. These algorithms prove to be of practical utility. For instance, they have been used in studies of the periodic Toda lattice, as well as, to study the inverse eigenvalue problems for Sturm-Liouville equations and Hill's equation. In particular, the results obtained by Ferguson allowed solving the IEP-G, when $G=C_{n}$.

In search to advance towards the knowledge of a solution of the IEP-G, a modified version of the IEP-G arises, known as the {\it $\lambda, \mu -$inverse eigenvalue problem of a graph $G$} ($\lambda, \mu -$IEP-G) \cite{Bar, Duarte}, which asks:
\begin{problem}[$\lambda, \mu -$IEP-G]\label{Q1}{\rm \cite{Duarte, Bar}}
Given a graph $G$ on $n$ vertices, $i\in V(G)$ and $2n-1$ real numbers satisfying the interlacing inequalities \
$\lambda_{1}\geq \mu_{1}\geq \lambda_{2}\geq \mu_{2}\geq \cdots \geq \mu_{n-1}\geq \lambda_{n}$, \
is there a matrix $M\in \mathcal{S}(G)$ with $\sigma(M)=\{\lambda_{1},\lambda_{2},\ldots,\lambda_{n}\}$ and $\sigma(M(i))=\{\mu_{1},\mu_{2},\ldots,\mu_{n-1}\}$?
\end{problem}

The $\lambda, \mu -$IEP-G allowed to solve the IEP-G when the set of real numbers $\{\lambda_{1}, \lambda_{2},\cdots,\lambda_{n}\}$ is chosen satisfying the condition
\begin{align}\label{Id1}
\lambda_{1}>\lambda_{2}>\cdots >\lambda_{n},
\end{align}
for the following graphs:
\\
$\blacktriangleright$ {\bf Trees $T$ and connected graphs} \cite{Duarte, Has}.
\\
$\blacktriangleright$ {\bf Complete graphs and small graphs (up to $4$ vertices)} \cite{Bar, Bar1}.
\\

Motivated by relaxing the condition \eqref{Id1}, the authors in \cite{Jh-Du} have characterized the possible lists of ordered multiplicities among matrices whose graph is a connected graph, for the following subfamilies of trees:
\\
$\blacktriangleright$ {\bf Generalized stars and double generalized stars} \cite{Jh-Du}.
\\
$\blacktriangleright$ {\bf Graphs with up to $5$ vertices:} Barret et al. \cite{Bar2} introduced the {\it strong spectral property} ({\it SSP}), which establishes:
\begin{definition}{\rm \cite{Bar2}}\label{ssp}
{\rm A matrix $A\in S_{n}(\mathbb{R})$ is said to have the {\it strong spectral property} ({\it SSP}) if the zero matrix $X=O$ is the only matrix $X\in S_{n}(\mathbb{R})$ that satisfies:
\begin{enumerate}
\item[$i$.] \ $[A,X] = AX-XA = O$, \ and
\item[$ii$.] \ $A\circ X = I\circ X = O$, \ where $\circ$ denotes the entry-wise product of matrices.
\end{enumerate}
}
\end{definition}
This property has turned out to be a powerful tool. Barret et al. \cite{Bar3} have introduced new techniques to the IEP-G, based on the SSP. They solved the IEP-G for graphs with up to 5 vertices, with the stipulation that there is a matrix realizing each possible spectrum that has the SSP. Considering a new technique utilizing the SSP, the authors in \cite{H.Lin} completely solve the IEP-G for a subfamily of clique-path graphs, in particular for the subfamilies of graphs: \\
$\blacktriangleright$ {\bf Lollipop graphs}. The {\it $(k,p)-$ lollipop graph} $L_{k,p}$ is the graph on $k+p$ vertices obtained by adding an edge between a vertex in a complete graph $K_{k}$ and a leaf of a path graph $P_{p}$. The graph $L_{4,3}$ is depicted below:
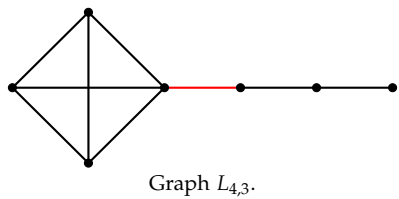
\begin{align*}
\begin{tikzpicture}
\draw [line width=0.8pt] (0.,0.)-- (1.,1.);
\draw [line width=0.8pt] (1.,1.)-- (1.,-1.);
\draw [line width=0.8pt] (1.,1.)-- (2.,0.);
\draw [line width=0.8pt] (1.,-1.)-- (2.,0.);
\draw [line width=0.8pt] (0.,0.)-- (1.,-1.);
\draw [line width=0.8pt] (0.,0.)-- (2.,0.);
\draw [line width=0.8pt][color=red] (2.,0.)-- (3.,0.);
\draw [line width=0.8pt] (3.,0.)-- (5.,0.);
\begin{scriptsize}
\draw [fill=black] (2.,0.) circle (1.5pt);
\draw [fill=black] (3.,0.) circle (1.5pt);
\draw [fill=black] (4.,0.) circle (1.5pt);
\draw [fill=black] (5.,0.) circle (1.5pt);
\draw [fill=black] (0.,0.) circle (1.5pt);
\draw [fill=black] (1.,1.) circle (1.5pt);
\draw [fill=black] (1.,-1.) circle (1.5pt);
\draw[color=black] (2.5,-1.3) node {Graph  $L_{4,3}$.};
\end{scriptsize}
\end{tikzpicture}
\end{align*}
For the family of lollipop graphs, the authors obtained in \cite{H.Lin}:
\vspace{-0.2cm}
\begin{theorem}[\cite{H.Lin}, Corollary 2.6]
Let $\sigma=\{\lambda_{1}, \cdots,\lambda_{k},\cdots,\lambda_{k+p}\}$ be a set with $k+p$ real numbers, where $k\geq 2$. Then $\sigma$ is the spectrum of a matrix $A\in \mathcal{S}(L_{k,p})$ ($A$ with the property SSP) if and only if $\sigma$ contains at least $p+2$ distinct elements.
\end{theorem}

Given two graphs $G$ and $H$, the disjoint union of $G$ and $H$ is denoted by $G\oplus H$. If each of $G$ and $H$ has a vertex labeled as $v$, then the {\it vertex-sum} \ $G\oplus_{v} H$ \ of $G$ and $H$ at $v$ is the graph obtained from $G\oplus H$ by identifying the two vertices labeled by $v$.
\\
$\blacktriangleright$ {\bf Generalized barbell graphs}. A {\it generalized barbell graph} on $k_{1}+p+k_{2}$ vertices is the graph defined as follows: $B_{k_{1},p,k_{2}}=K_{k_{1}}\oplus_{v} P_{p+2}\oplus_{w} K_{k_{2}}$, where $v$ and $w$ are the leaves of $P_{p+2}$. An example of the generalized barbell graph $B_{4,3,6}$ is depicted below:
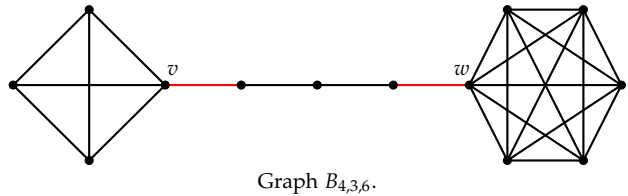
\begin{align*}
\begin{tikzpicture}
\draw [line width=0.8pt] (0.,0.)-- (1.,1.);
\draw [line width=0.8pt] (1.,1.)-- (1.,-1.);
\draw [line width=0.8pt] (1.,1.)-- (2.,0.);
\draw [line width=0.8pt] (1.,-1.)-- (2.,0.);
\draw [line width=0.8pt] (0.,0.)-- (1.,-1.);
\draw [line width=0.8pt] (0.,0.)-- (2.,0.);
\draw [line width=0.8pt][color=red] (2.,0.)-- (3.,0.);
\draw [line width=0.8pt] (3.,0.)-- (5.,0.);
\draw [line width=0.8pt][color=red] (5.,0.)-- (6.,0.);
\draw [line width=0.8pt] (6.5,1.)-- (7.5,1.);
\draw [line width=0.8pt] (7.5,1.)-- (8.,0.);
\draw [line width=0.8pt] (8.,0.)-- (7.5,-1.);
\draw [line width=0.8pt] (7.5,-1.)-- (6.5,-1.);
\draw [line width=0.8pt] (6.5,-1.)-- (6.,0.);
\draw [line width=0.8pt] (6.,0.)-- (6.5,1.);
\draw [line width=0.8pt] (6.5,1.)-- (8.,0.);
\draw [line width=0.8pt] (7.5,1.)-- (6.,0.);
\draw [line width=0.8pt] (6.5,1.)-- (7.5,-1.);
\draw [line width=0.8pt] (7.5,1.)-- (6.5,-1.);
\draw [line width=0.8pt] (6.5,1.)-- (6.5,-1.);
\draw [line width=0.8pt] (7.5,-1.)-- (7.5,1.);
\draw [line width=0.8pt] (6.,0.)-- (8.,0.);
\draw [line width=0.8pt] (6.5,-1.)-- (8.,0.);
\draw [line width=0.8pt] (7.5,-1.)-- (6.,0.);
\begin{scriptsize}
\draw [fill=black] (2.,0.) circle (1.5pt);
\draw[color=black] (2.1,0.2) node {$v$};
\draw [fill=black] (3.,0.) circle (1.5pt);
\draw [fill=black] (4.,0.) circle (1.5pt);
\draw [fill=black] (5.,0.) circle (1.5pt);
\draw [fill=black] (0.,0.) circle (1.5pt);
\draw [fill=black] (1.,1.) circle (1.5pt);
\draw [fill=black] (1.,-1.) circle (1.5pt);
\draw [fill=black] (6.,0.) circle (1.5pt);
\draw[color=black] (5.9,0.2) node {$w$};
\draw [fill=black] (8.,0.) circle (1.5pt);
\draw [fill=black] (6.5,1.) circle (1.5pt);
\draw [fill=black] (7.5,1.) circle (1.5pt);
\draw [fill=black] (6.5,-1.) circle (1.5pt);
\draw [fill=black] (7.5,-1.) circle (1.5pt);
\draw[color=black] (4.,-1.3) node {Graph  $B_{4,3,6}$.};
\end{scriptsize}
\end{tikzpicture}
\end{align*}
For the family of generalized barbell graphs, the authors obtained in \cite{H.Lin}:
\vspace{-0.2cm}
\begin{theorem}[\cite{H.Lin}, Corollary 2.8]
Let $\sigma=\{\lambda_{1}, \cdots,\lambda_{k_{1}},\cdots,\lambda_{k_{1}+p},\cdots,\lambda_{k_{1}+p+k_{2}}\}$ be a set with $k_{1}+p+k_{2}$ real numbers, where $k_{1}\geq 2$ and $k_{2}\geq 2$. Then $\sigma$ is the spectrum of a matrix $A\in \mathcal{S}(B_{k_{1},p,k_{2}})$ ($A$ with the property SSP) if and only if $\sigma$ contains at least $p+4$ distinct elements.
\end{theorem}

Consequently, many spectra are shown realizable for {\it block graphs}. A {\it block} of a graph is its maximal $2-$connected induced subgraph. A {\bf block graph} $G$ is a graph whose blocks are cliques (for more details on this structure, see \cite{H.Lin}). \\

For $A\in S_{n}(\mathbb{R})$, we denote by $q(A)$ the number of distinct eigenvalues of $A$. For a graph $G$, we define
\begin{align*}
q(G)=min\{q(A) \ : \ A\in \mathcal{S}(G)\}.
\end{align*}
A known result given in \cite{Ahmadi} is the following:
\begin{proposition}{\rm\cite{Ahmadi}}\label{prop1}
Suppose $G$ is obtained from $K_{n}$ by deleting a single edge $e$. Then
\begin{align*}
q(G)=\begin{cases}
1, \ \ if \ \ n = 2; \\
3, \ \ if \ \ n = 3; \\
2, \ \ otherwise.
\end{cases}
\end{align*}
\end{proposition}

In this study the following results will play a crucial role:
\begin{lemma}{\rm \cite{Bol2}}\label{L1}
Let $\{\lambda_{1}, \lambda_{2},\cdots,\lambda_{n}\}$ and $\{\mu_{1},\mu_{2},\cdots,\mu_{n-1}\}$ be two sets of real numbers such that
    \begin{align}\label{Id2}
    \lambda_{1}\geq \mu_{1}\geq \lambda_{2}\geq \mu_{2}\geq \cdots \geq \mu_{n-1}\geq \lambda_{n},
    \end{align}
where $\mu_{i}\neq \mu_{j}$ for all $i\neq j$. Then there exists an $n\times n$ bordered matrix \
$A = \begin{bmatrix}
A_{1} & \mathbf{b} \\
\mathbf{b}^{T} & a
\end{bmatrix}$ \
with eigenvalues $\{\lambda_{1}, \lambda_{2},\cdots,\lambda_{n}\}$, where $A_{1}=diag\{\mu_{1},\mu_{2},\cdots,\mu_{n-1}\}$
and $\mathbf{b}=\begin{bmatrix} b_{1} \ \cdots \ b_{n-1}\end{bmatrix}^{T}$ is such that
\begin{align*}
b^{2}_{i}=-\frac{\displaystyle\prod_{j=1}^{n}(\mu_{i}-\lambda_{j})}{\displaystyle\prod_{\substack{j=1 \\ j\neq i}}^{n-1}(\mu_{i}-\mu_{j})}, \quad i=1,\ldots,n-1.
\end{align*}
\end{lemma}

\begin{remark}\label{Rem1}
Note that Lemma \ref{L1} answers Problem \ref{Q1} when G is a star, if all the inequalities in \eqref{Id2} are strict.
\end{remark}

\begin{lemma}{\rm \cite{smi}}\label{smi1}
Let $B\in S_{k}(\mathbb{R})$ be a matrix with eigenvalues $\mu_{1},\ldots,\mu_{k}$ and let $\mathbf{u}$ be an eigenvector corresponding to $\mu_{1}$, normalized so that $\mathbf{u}^{T}\mathbf{u}=1$. Let \
$A=\begin{bmatrix}
A_{1} & \mathbf{b} \\
\mathbf{b}^{T} & \mu_{1}
\end{bmatrix}\in S_{n}(\mathbb{R}) $ \ be a matrix with a diagonal element $\mu_{1}$ and eigenvalues $\lambda_{1},\ldots,\lambda_{n}$. Then the matrix \
$C=\begin{bmatrix}
A_{1} & \mathbf{b}\mathbf{u}^{T} \\
\mathbf{u}\mathbf{b}^{T} & B
\end{bmatrix}\in S_{n+k-1}(\mathbb{R})$, \ has eigenvalues $\lambda_{1},\ldots,\lambda_{n},\mu_{2},\ldots,\mu_{k}$.
\end{lemma}

\section{Main results}
In this section we present our main results. Before, we need an important result, which will play a crucial role. We begin by considering the following nonsingular matrix:
\begin{align}\label{I1}
R=\left[
\begin{array}{ccccc}
1 & 1 & \cdots & 1 & 1 \\
1 & 1 & \cdots & 1 & -1 \\
\vdots & \vdots & \cdots & -2 & 0 \\
1 & 1 & \ddots & \ddots & \vdots \\
1 & -(n-1) & 0 & \cdots & 0%
\end{array}.%
\right]
\end{align}

In the next result, we construct a matrix in $S(K_{n})$ with prescribed spectrum $\Lambda =\{\lambda
_{1},\lambda _{2},\ldots ,\lambda _{n}\}$ satisfying $\lambda _{1}\neq \lambda_{2}$, which satisfies some important properties:
\begin{theorem}\label{Rob1} Let $\Lambda =\{\lambda _{1},\lambda _{2},\ldots ,\lambda _{n}\}$
be a list of real numbers with at least $2$ distinct elements. Then, there exists a matrix $A=(a_{ij})$ of order $n\times n$, which satisfies:
\begin{itemize}
\item[i.] $\sigma(A)=\Lambda$
\item[ii.] $A\in \mathcal{S}(K_{n})$
\item[iii.] $A\in \mathcal{CS}_{\lambda_{1}}$. Furthermore, for all $k=2,\ldots,n$, $(\lambda_{k}, {\bf x}_{k})$ is a eigenpair for $A$, being
\begin{align*}
{\bf x}^{T}_{k}=(1,1,\ldots ,1,-(n-k+1),\underset{(k-2)\text{ zeros}}{\underbrace{0,\ldots ,0}}).
\end{align*}
\item[iv.] $a_{nn}=\left(\frac{1}{n}\right)(\lambda_{1}-\lambda_{2})+\lambda_{2}$.
\end{itemize}
\end{theorem}
\begin{proof}\
$i.$ Let $D=diag\{\lambda _{1},\lambda _{2},\ldots,\lambda _{n}\}$, where $\lambda _{1}\neq \lambda _{2}$. Let us define the matrix
\begin{align*}
A=RDR^{-1}
\end{align*}
where $R$ is as in \eqref{I1}. Then, it is clear that $\sigma(A)=\Lambda$.
\\ \\
$ii.$ \
We now show that $A\in \mathcal{S}(K_{n})$. First, we show that $A$ is symmetric. In fact, from \eqref{I1} we obtain that the rows $\mathbf{G}_{i}$ of $R$ are defined by
\begin{eqnarray*}
\mathbf{G}_{1} &=&(1,1,\ldots ,1) \\ \\
\mathbf{G}_{i} &=&(1,1,\ldots ,1,-(i-1),\underset{(i-2)\text{ zeros}}{%
\underbrace{0,\ldots ,0}}) ,\text{ }i=2,\ldots ,n,
\end{eqnarray*}
while the columns $\mathbf{H}_{j}$ of $R^{-1}=\left[ \mathbf{H}_{1}\mid
\cdots \mid \mathbf{H}_{n}\right] $ are defined by%
\begin{eqnarray*}
\mathbf{H}_{1}^{T} &=&\left(\frac{1}{n},\frac{1}{n(n-1)},\ldots ,\frac{1}{%
3\cdot 2},\frac{1}{2\cdot 1}\right) \\ \\
\mathbf{H}_{j}^{T} &=&\left( \frac{1}{n},\frac{1}{n(n-1)},\ldots ,\frac{1}{%
(j+1)j},-\frac{1}{j},\underset{(j-2)\text{ zeros}}{\underbrace{0,\ldots ,0}}%
\right) ,\text{ }j=2,\ldots ,n.
\end{eqnarray*}
Then, for $i,j=1,2,\ldots,n$,
\begin{equation*}
a_{i,j}=(RDR^{-1})_{i,j}=\mathbf{e}_{i}^{T}(RDR^{-1})\mathbf{e}_{j}=%
\mathbf{G}_{i}D\mathbf{H}_{j}.
\end{equation*}%
Thus, when $i=1$ and $j=2,\ldots ,n$, we have
\begin{align}\label{sim1}
a_{1,j} &= \mathbf{G}_{1}D\mathbf{H}_{j} = \left(\lambda _{1},\lambda _{2},\ldots,\lambda _{n}\right)
\mathbf{H}_{j} \notag \\ \notag \\
&= \frac{\lambda _{1}}{n}+\frac{\lambda _{2}}{n(n-1)}+\cdots +
\frac{\lambda_{n-(j-1)}}{(j+1)j}-\frac{\lambda_{n-(j-2)}}{j} \\ \notag \\
&=\mathbf{G}_{j}D\mathbf{H}_{1}=a_{j,1}. \notag
\end{align}
Similarly, when $2\leq i<j\leq n$, we have
\begin{align}\label{sim2}
a_{i,j} &= \mathbf{G}_{i}D\mathbf{H}_{j} = (\lambda_{1}, \lambda_{2}, \ldots, \lambda_{n-(i-1)}, -(i-1)\lambda_{n-(i-2)},
\underset{(i-2)\text{ zeros}}{\underbrace{0, \ldots, 0}})\mathbf{H}_{j} \notag \\
&= \ \frac{\lambda_{1}}{n}+\frac{\lambda_{2}}{n(n-1)}+\cdots +
\frac{\lambda_{n-(j-1)}}{(j+1)j}-\frac{\lambda_{n-(j-2)}}{j}
\\ \notag \\ &= \ (\lambda_{1}, \lambda_{2}, \ldots, \lambda_{n-(j-1)}, -(j-1)\lambda_{n-(j-2)},
\underset{(j-2)\text{ zeros}}{\underbrace{0, \ldots, 0}})\mathbf{H}_{i} \notag \\
&= \ \mathbf{G}_{j}D\mathbf{H}_{i} \ = \ a_{j,i}. \notag
\end{align}
Thus, we have showed that $A$ is symmetric. On the other hand, from \eqref{sim1} and \eqref{sim2}, we can conclude that
\begin{align}\label{sim3}
a_{1,j}=a_{i,j}, \quad \text{for all} \quad 2\leq i<j\leq n.
\end{align}
Note also that \eqref{sim1} can be written by
\begin{align*}
a_{1,j} &= \frac{\lambda_{1}-\lambda_{2}}{n}+\frac{\lambda_{2}-\lambda_{3}}{n-1}+\cdots +
\frac{\lambda_{n-j}-\lambda_{n-(j-1)}}{j+1}+\frac{\lambda_{n-(j-1)}-\lambda_{n-(j-2)}}{j} \\ \\
&=a_{1,n}+\varphi(j);
\end{align*}
where
\begin{align*}
a_{1,n}=\frac{\lambda_{1}-\lambda_{2}}{n}\neq 0 \quad \ \text{and} \ \quad \varphi(j)=\sum\limits_{s=j}^{n-1}\frac{\lambda_{n-(s-1)}-\lambda_{n-(s-2)}}{s}, \ j=2,\ldots ,n-1.
\end{align*}
Note that if $\lambda_{i}\in \{\lambda_{1}, \lambda_{2}\}$, for all $i=3,\ldots ,n$ (at this case, $\Lambda$ has only $2$ distinct elements), then $\varphi(j)=\alpha a_{1,n}$, for some $\alpha\in \mathbb{R}-\{-1\}$, for all $j=2,\ldots ,n-1$. Therefore, $a_{1,j}=\beta a_{1,n}$, for some $\beta\in \mathbb{R}-\{0\}$, for all $j=2,\ldots ,n-1$. Thus, by \eqref{sim3}, $A\in S(K_{n})$.

Let us then assume that $\Lambda$ has at least three distinct elements. At this case, it is suffice to choose any permutation $\phi$ of the set $\{2,\ldots,n-1\}$ for which
\[\varphi(j)=\sum\limits_{s=j}^{n-1}\frac{\lambda_{n-(\phi(s)-1)}-\lambda_{n-(\phi(s)-2)}}{s}\neq -a_{1,n}, \quad \text{for all} \quad j=2,\ldots ,n-1.\] Thus, $a_{1,j}=a_{1,n}+\varphi(j)\neq 0$, for all $j=2,\ldots ,n-1$ and by \eqref{sim3}, $A\in S(K_{n})$.
\\ \\
$iii.$ \ Since $R{\bf e}_{1}={\bf e}$, \ $R^{-1}{\bf e}={\bf e}_{1}$. Therefore,
\begin{align*}
A{\bf e}=(RDR^{-1}){\bf e}=RD{\bf e}_{1}=\lambda_{1}R{\bf e}_{1}=\lambda_{1}{\bf e}.
\end{align*}
Thus, $A\in \mathcal{CS}_{\lambda_{1}}$. Similarly, we observe that $R{\bf e}_{k}={\bf x}_{k}$, where
\begin{align*}
{\bf x}^{T}_{k}=(1,1,\ldots ,1,-(n-k+1),\underset{(k-2)\text{ zeros}}{\underbrace{0,\ldots ,0}}), \quad k=2,\ldots,n.
\end{align*}
Hence, $R^{-1}{\bf x}_{k}={\bf e}_{k}$. Therefore,
\begin{align*}
A{\bf x}_{k}=(RDR^{-1}){\bf x}_{k}=RD{\bf e}_{k}=\lambda_{k}R{\bf e}_{k}=\lambda_{k}{\bf x}_{k}.
\end{align*}
Thus, $(\lambda_{k}, {\bf x}_{k})$ is a eigenpair for $A$.
\\ \\
$iv.$ \ \begin{align*}
a_{n,n}=\mathbf{G}_{n}D\mathbf{H}_{n}&=(\lambda_{1},-(n-1)\lambda_{2},\underset{(n-2)\text{ zeros}}{%
\underbrace{0,\ldots ,0}})\mathbf{H}_{n} \\
&=\frac{\lambda_{1}}{n}+\frac{(n-1)\lambda_{2}}{n} \\
&=\left(\frac{1}{n}\right)(\lambda_{1}-\lambda_{2})+\lambda_{2}.
\end{align*}
The proof is complete.
\end{proof}

\begin{theorem}\label{main1}
Let $G$ be a graph of order $n$ having a clique of order $k$ and a cluster $(C,S)$, where $\vert C\vert=n-k$ and $\vert S\vert=r\leq k$. For any list $\Lambda =\{\lambda _{1},\lambda _{2},\ldots ,\lambda _{n}\}$ of real numbers with at least $n-k+1$ distinct elements such that at least two of them have algebraic multiplicity greater than or equal to $2$, there exists a matrix $M\in \mathcal{S}(G)$ with spectrum $\Lambda$.
\end{theorem}
\begin{proof}
Let us write $\Lambda=\Lambda_{1}\cup \Lambda_{2}$, where
\begin{align*}
\Lambda_{1}=\{\lambda _{1,1},\lambda _{1,2},\ldots ,\lambda _{1,n-k+1}\},
\end{align*}
is a list whose elements are arranged in strictly decreasing order. We choose $n-k$ real numbers $\mu_{j}$, $j=1,\ldots,n-k$, satisfying
\begin{align*}
\lambda _{1,1}>\mu_{1}>\lambda _{1,2}>\mu_{2}>\lambda_{1,3}>\cdots >\lambda _{1,n-k}>\mu_{n-k}>\lambda _{1, n-k+1},
\end{align*}
such that $a=\left(\lambda _{1,1}+\sum\limits_{j=1}^{n-k}(\mu_{j}-\lambda _{1,j+1})\right)\notin \Lambda_{2}$. From Lemma \ref{L1} and Remark \ref{Rem1}, we can construct a bordered matrix $A=\begin{bmatrix}A_{1} & {\bf b} \\ {\bf b}^{T} & a\end{bmatrix}\in \mathcal{S}(S_{n-k+1})$ with spectrum $\Lambda_{1}$, where $A_{1}=diag\{\mu_{1},\ldots,\mu_{n-k}\}$ and $a=trace(A)-trace(A_{1})$. Now, we consider the list $\Gamma=\Lambda_{2}\cup \{a\}$. Since $\Lambda$ has at least two distinct elements with algebraic multiplicity greater than or equal to $2$ and $a\notin \Lambda_{2}$, we obtain that $\Gamma$ has at least three distinct eigenvalues. Taking into account the proof of Theorem \ref{Rob1}, we construct the diagonal matrix $D$ whose diagonal entries are the elements of $\Gamma$, where $a$ is its $(k-r+2,k-r+2)-$th position. Thus, from Theorem \ref{Rob1} ($iii.$), we get that the matrix $B=RDR^{-1}$ has an eigenvector ${\bf u}$ with $(k-r)-$zeros entries associated to $a$. Finally, Lemma \ref{smi1} allows us conclude that
\begin{align*}
M=\begin{bmatrix}
A_{1} & \mathbf{b}\mathbf{u}^{T} \\
\mathbf{u}\mathbf{b}^{T} & B
\end{bmatrix}\in \mathcal{S}(G),
\end{align*}
and has spectrum $\Lambda$.
\end{proof}

\begin{remark}\
\begin{itemize}
\item[1.] If in the previous theorem the list $\Lambda$ has more than $n-k+1$ distinct elements, then the condition that $\Lambda$ has at least two of elements with algebraic multiplicity greater than or equal to $2$ is not necessary.
\item[2.] When $G$ is a star, Theorem \ref{main1} is consistent with Remark \ref{Rem1} and allows us to construct immediately an $n\times n$ bordered matrix solution.
\end{itemize}
\end{remark}

An interesting conclusion derived from the previous theorem is the following result:
\begin{corollary}
Let $G=K_{n}-\{e\}$. Any list $\Lambda =\{\lambda _{1},\lambda _{2},\ldots ,\lambda _{n}\}$ with at least three distinct real numbers is the spectrum of a matrix $M\in \mathcal{S}(G)$. Also, when $n=3$, $M\in \mathcal{S}(G)$ if and only if $\sigma(M)$ has exactly three distinct elements.
\end{corollary}
\begin{proof}
The proof follows taking into account Theorem \ref{main1} and Proposition \ref{prop1}.
\end{proof}

\begin{example} Let $G$ be the graph given in Example \ref{Ex1}. Let $\Lambda=\{7^{[2]},1^{[2]},-3^{[3]},-5^{[2]}\}$. Following the proof from Theorem \ref{main1}, we construct a matrix $C\in \mathcal{S}(G)$ with spectrum $\Lambda$. In fact, let us partition $\Lambda$ as follows:
\begin{align*}
\Lambda=\{7,1,-3,-5\}\cup \{7,1,-3,-3,-5\}=\Lambda_{1}\cup \Lambda_{2}.
\end{align*}
Let us choose the real numbers $\mu_{1}=6$, $\mu_{2}=0$ and $\mu_{3}=-4$. Thus, $a=-2$. From Lemma \ref{L1} and Remark \ref{Rem1}, we get a bordered matrix
\begin{align*}
A=\begin{bmatrix}
6 & 0 & 0 & -\frac{\sqrt{33}}{2} \\ \\
0 & 0 & 0 & \frac{\sqrt{70}}{4} \\ \\
0 & 0 & -4 & -\frac{\sqrt{22}}{4} \\ \\
-\frac{\sqrt{33}}{2} & \frac{\sqrt{70}}{4} & -\frac{\sqrt{22}}{2} & -2
\end{bmatrix}=\begin{bmatrix}A_{1} & {\bf b} \\ {\bf b}^{T} & a\end{bmatrix}\in \mathcal{S}(S_{4})
\end{align*}
with spectrum $\Lambda_{1}$. Now, we consider the list $\Gamma=\Lambda_{2}\cup \{-2\}$. Since we need a matrix $B$ with spectrum $\Gamma$ and a eigenvector ${\bf u}$ with $4$-zeros entries associated to $-2$, then we choose the diagonal matrix $D=diag\{7,1,-3,-3,-5,-2\}$ and thus we get that
\begin{align*}
B=RDR^{-1}=\begin{bmatrix}
-\frac{31}{30} & \frac{29}{30} & \frac{37}{15} & \frac{9}{5} & \frac{9}{5} & 1 \\ \\
\frac{29}{30} & -\frac{31}{30} & \frac{37}{15} & \frac{9}{5} & \frac{9}{5} & 1 \\ \\
\frac{37}{15} & \frac{37}{15} & -\frac{38}{15} & \frac{9}{5} & \frac{9}{5} & 1 \\ \\
\frac{9}{5} & \frac{9}{5} & \frac{9}{5} & -\frac{6}{5} & \frac{9}{5} & 1 \\ \\
\frac{9}{5} & \frac{9}{5} & \frac{9}{5} & \frac{9}{5} & -\frac{6}{5} & 1 \\ \\
1 & 1 & 1 & 1 & 1 & 2
\end{bmatrix}
\end{align*}
is such that $B{\bf u}=-2{\bf u}$, where ${\bf u}^{T}=\frac{1}{\sqrt{2}}\begin{bmatrix}1 & -1 & 0 & 0 & 0 & 0 \end{bmatrix}$. Finally, applying Lemma \ref{smi1}, we obtain the matrix
\begin{align*}
C=\begin{bmatrix}
A_{1} & \mathbf{b}\mathbf{u}^{T} \\
\mathbf{u}\mathbf{b}^{T} & B
\end{bmatrix}={\normalsize \begin{bmatrix}
6 & 0 & 0 & -\frac{\sqrt{66}}{4} & \frac{\sqrt{66}}{4} & 0 & 0 & 0 & 0 \\ \\
0 & 0 & 0 & \frac{\sqrt{35}}{4} & -\frac{\sqrt{35}}{4} & 0 & 0 & 0 & 0 \\ \\
0 & 0 & -4 & -\frac{\sqrt{11}}{4} & \frac{\sqrt{11}}{4} & 0 & 0 & 0 & 0 \\ \\
-\frac{\sqrt{66}}{4} & \frac{\sqrt{35}}{4} & -\frac{\sqrt{11}}{4} & -\frac{31}{30} & \frac{29}{30} & \frac{37}{15} & \frac{9}{5} & \frac{9}{5} & 1 \\ \\
\frac{\sqrt{66}}{4} & -\frac{\sqrt{35}}{4} & \frac{\sqrt{11}}{4} & \frac{29}{30} & -\frac{31}{30} & \frac{37}{15} & \frac{9}{5} & \frac{9}{5} & 1 \\ \\
0 & 0 & 0 & \frac{37}{15} & \frac{37}{15} & -\frac{38}{15} & \frac{9}{5} & \frac{9}{5} & 1 \\ \\
0 & 0 & 0 & \frac{9}{5} & \frac{9}{5} & \frac{9}{5} & -\frac{6}{5} & \frac{9}{5} & 1 \\ \\
0 & 0 & 0 & \frac{9}{5} & \frac{9}{5} & \frac{9}{5} & \frac{9}{5} & -\frac{6}{5} & 1 \\ \\
0 & 0 & 0 & 1 & 1 & 1 & 1 & 1 & 2
\end{bmatrix}}\in \mathcal{S}(G),
\end{align*}
which has spectrum $\Lambda$.
\end{example}

\begin{theorem}\label{main3}
Let $G$ be a graph of order $n$ having a clique of order $k$ and a cluster $(C,S)$, where $\vert C\vert=n-k$ and $\vert S\vert=r\leq k$. For any list $\Lambda =\{\lambda _{1},\lambda _{2},\ldots ,\lambda _{n}\}$ of real numbers with at least $3$ distinct elements such that at least two of them have algebraic multiplicity greater than or equal to $2$, there exists a matrix $M\in \mathcal{S}(G(K_{n-k}))$ with spectrum $\Lambda$.
\end{theorem}
\begin{proof}
Let us write $\Lambda=\Lambda_{1}\cup \Lambda_{2}$, where
\begin{align*}
\Lambda_{1}=\{\lambda _{1,1},\lambda _{1,2},\ldots ,\lambda _{1,n-k+1}\} \quad \text{and} \quad \Lambda_{2}=\Lambda-\Lambda_{1},
\end{align*}
with $\Lambda_{1}$ and $\Lambda_{2}$ having at least two distinct elements and such that
\begin{align*}
\frac{1}{n}(\lambda _{1,1}-\lambda _{1,2})+\lambda _{1,2}\notin \Lambda_{2}; \quad \quad \lambda _{1,1}\neq \lambda _{1,2}.
\end{align*}
From Theorem \ref{Rob1}, we consider the matrix $R$ (as in \eqref{I1}) of order $n-k+1$ and we construct the matrix $A=RD_{1}R^{-1}=(a_{ij})$, where $D_{1}$ is the diagonal matrix whose diagonal entries are the elements of $\Lambda_{1}$ and $(D_{1})_{n-k+1,n-k+1}=a=\frac{1}{n}(\lambda _{1,1}-\lambda _{1,2})+\lambda _{1,2}$. On the other hand, let us consider the list $\Gamma=\Lambda_{2}\cup \{a\}$. Since $a=\frac{1}{n}(\lambda _{1,1}-\lambda _{1,2})+\lambda _{1,2}\notin \Lambda_{2}$, $\Gamma$ has at least three distinct eigenvalues. Taking into account the proof of Theorem \ref{Rob1}, we construct the diagonal matrix $D_{2}$ whose diagonal entries are the elements of $\Gamma$, where $(D_{2})_{k-r+2,k-r+2}=a$. Thus, from Theorem \ref{Rob1} ($iii.$), we get that the matrix $B=RD_{2}R^{-1}$ has an eigenvector ${\bf u}$ with $(k-r)-$zeros entries associated to $a$. Finally, Lemma \ref{smi1} allows us conclude that
\begin{align*}
M=\begin{bmatrix}
A_{1} & \mathbf{b}\mathbf{u}^{T} \\
\mathbf{u}\mathbf{b}^{T} & B
\end{bmatrix}\in \mathcal{S}(G(K_{n-k})),
\end{align*}
and has spectrum $\Lambda$.
\end{proof}

\begin{example} Let $G(K_{3})$ the graph given in Example \ref{Ex2}. Let $\Lambda=\{7,-3^{[4]},-5^{[4]}\}$. We want to construct a matrix $C\in \mathcal{S}(G(K_{3}))$ with spectrum $\Lambda$. In fact, let us partition $\Lambda$ as follows:
\begin{align*}
\Lambda=\{7,-3^{[3]}\}\cup \{-3,-5^{[4]}\}=\Lambda_{1}\cup \Lambda_{2}.
\end{align*}
Note that $\frac{1}{4}(7-(-3))+(-3)=-\frac{1}{2}\notin \Lambda_{2}$. We construct the matrix
\begin{align*}
A=RD_{1}R^{-1}=\frac{1}{2}\begin{bmatrix}
-1 & 5 & 5 & 5 \\
5 & -1 & 5 & 5 \\
5 & 5 &	-1 & 5 \\
5 & 5 & 5 & -1
\end{bmatrix}=\begin{bmatrix}
A_{1} & {\bf b} \\
{\bf b}^{T} & -\frac{1}{2}
\end{bmatrix},
\end{align*}
where $D_{1}=diag\{7,-3,-3,-3\}$. We now consider the list $\Gamma=\Lambda_{2}\cup \{-\frac{1}{2}\}$. We need to construct a matrix $B$ with spectrum $\Gamma$ and with an eigenvector ${\bf u}$ having $k-r=4$-zeros entries associated to $-\frac{1}{2}$. Then, we choose the diagonal matrix $D_{2}=diag\{-3,-5,-5,-5,-5,-\frac{1}{2}\}$, where $(D_{2})_{6,6}=-\frac{1}{2}$. Thus, we have
\begin{align*}
B=RD_{2}R^{-1}=\frac{1}{12}\begin{bmatrix}
-29 & -23 & 4 & 4 & 4 & 4 \\
-23 & -29 & 4 & 4 & 4 & 4 \\
4 & 4 &	-56 & 4 & 4 & 4 \\
4 & 4 & 4 & -56 & 4 & 4 \\
4 & 4 &	4 & 4 & -56 & 4 \\
4 & 4 & 4 & 4 & 4 & -56
\end{bmatrix},
\end{align*}
which is such that $B{\bf u}=-\frac{1}{2}{\bf u}$, where ${\bf u}^{T}=\frac{1}{\sqrt{2}}\begin{bmatrix}-1 & 1 & 0 & 0 & 0 & 0 \end{bmatrix}$. Finally, applying Lemma \ref{smi1}, we obtain the matrix
\begin{align*}
C=\begin{bmatrix}
A_{1} & \mathbf{b}\mathbf{u}^{T} \\
\mathbf{u}\mathbf{b}^{T} & B
\end{bmatrix}={\normalsize\begin{bmatrix}
-\frac{1}{2} & \frac{5}{2} & \frac{5}{2} & \frac{5\sqrt{2}}{4} & -\frac{5\sqrt{2}}{4} & 0 & 0 & 0 & 0 \\ \\
\frac{5}{2} & -\frac{1}{2} & \frac{5}{2} & \frac{5\sqrt{2}}{4} & -\frac{5\sqrt{2}}{4} & 0 & 0 & 0 & 0 \\ \\
\frac{5}{2} & \frac{5}{2} & -\frac{1}{2} & \frac{5\sqrt{2}}{4} & -\frac{5\sqrt{2}}{4} & 0 & 0 & 0 & 0 \\ \\
\frac{5\sqrt{2}}{4} & \frac{5\sqrt{2}}{4} & \frac{5\sqrt{2}}{4} & -\frac{29}{12} & -\frac{23}{12} & \frac{1}{3} & \frac{1}{3} & \frac{1}{3} & \frac{1}{3} \\ \\
-\frac{5\sqrt{2}}{4} & -\frac{5\sqrt{2}}{4} & -\frac{5\sqrt{2}}{4} & -\frac{23}{12} & -\frac{29}{12} & \frac{1}{3} & \frac{1}{3} & \frac{1}{3} & \frac{1}{3} \\ \\
0 & 0 & 0 & \frac{1}{3} & \frac{1}{3} & -\frac{14}{3} & \frac{1}{3} & \frac{1}{3} & \frac{1}{3} \\ \\
0 & 0 & 0 & \frac{1}{3} & \frac{1}{3} & \frac{1}{3} & -\frac{14}{3} & \frac{1}{3} & \frac{1}{3} \\ \\
0 & 0 & 0 & \frac{1}{3} & \frac{1}{3} & \frac{1}{3} & \frac{1}{3} & -\frac{14}{3} & \frac{1}{3} \\ \\
0 & 0 & 0 & \frac{1}{3} & \frac{1}{3} & \frac{1}{3} & \frac{1}{3} & \frac{1}{3} & -\frac{14}{3}
\end{bmatrix}}\in \mathcal{S}(G(K_{3})),
\end{align*}
which has spectrum $\Lambda$.
\end{example}

Let $n$ and $j$ be given positive integers, with {\color{red}$j\leq n-2$}, and consider the family of graphs:
\begin{align}\label{Gi}
\mathcal{G}(i)=K_{i} \vee (K_{j}\cup K_{n-i-j}), \quad 1\leq i\leq n-j-1.
\end{align}
Labeling the vertices of $\mathcal{G}(i)$ starting with the vertices of $K_{j}$, continuing with the vertices of $K_{i}$ and finishing with the vertices of $K_{n-i-j}$, we obtain the following result, which can be seen as a particular case of Theorem \ref{main3}:
\begin{corollary}\label{main2}
Let $\mathcal{G}(i)$ be the family of graphs defined in \eqref{Gi}. For any list $\Lambda =\{\lambda _{1},\lambda _{2},\ldots ,\lambda _{n}\}$ of real numbers with at least three distinct elements such that at least two of them have algebraic multiplicity greater than or equal to $2$, there exists a matrix $M\in \mathcal{S}(\mathcal{G}(i))$ with spectrum $\Lambda$.
\end{corollary}
\begin{proof}
It is follows from Theorem \ref{main3}, by considering $G_{i}$, with $1\leq i\leq n-j-1$, as the graph having a clique of order $n-j$ and a cluster $(C,S)$, where $\vert C\vert=j$ and $\vert S\vert=i$. Thus, for each $1\leq i\leq n-j-1$, $\mathcal{G}(i)$ can be seen as the graph $G_{i}(K_{j})$.
\end{proof}

\begin{example} Let $n=9$, $j=2$ and $\mathcal{G}(4)=K_{4}\vee (K_{2}\cup K_{3})$ be the graph depicted in Figure 2:
\begin{center}
\begin{tikzpicture}
\draw [line width=0.5pt] (0.,1.)-- (2.,1.);
\draw [line width=0.5pt] (0.,1.)-- (4.,2.);
\draw [line width=0.5pt] (0.,1.)-- (4.,-2.);
\draw [line width=0.5pt] (0.,1.)-- (2.,-1.);
\draw [line width=0.5pt] (0.,-1.)-- (4.,2.);
\draw [line width=0.5pt] (0.,-1.)-- (4.,-2.);
\draw [line width=0.5pt] (0.,-1.)-- (2.,-1.);
\draw [line width=0.5pt] (4.,2.)-- (6.,0.);
\draw [line width=0.5pt] (4.,2.)-- (7.,-2.);
\draw [line width=0.5pt] (4.,2.)-- (7.,2.);
\draw [line width=0.5pt] (2.,1.)-- (6.,0.);
\draw [line width=0.5pt] (2.,1.)-- (7.,-2.);
\draw [line width=0.5pt] (2.,1.)-- (7.,2.);
\draw [line width=1.pt] (2.,1.)-- (4.,-2.);
\draw [line width=0.5pt] (4.,-2.)-- (6.,0.);
\draw [line width=0.5pt] (4.,-2.)-- (7.,-2.);
\draw [line width=0.5pt] (4.,-2.)-- (7.,2.);
\draw [line width=0.5pt] (2.,-1.)-- (7.,-2.);
\draw [line width=0.5pt] (2.,-1.)-- (7.,2.);
\draw [line width=1.pt] (2.,-1.)-- (4.,2.);
\draw [line width=0.5pt] (2.,1.)-- (0.,-1.);
\draw [line width=0.5pt] (6.,0.)-- (2.,-1.);
\draw [line width=1.pt] (7.,2.)-- (6.,0.);
\draw [line width=1.pt] (7.,2.)-- (7.,-2.);
\draw [line width=1.pt] (6.,0.)-- (7.,-2.);
\draw [line width=1.pt] (0.,1.)-- (0.,-1.);
\draw [line width=1.pt] (4.,2.)-- (2.,1.);
\draw [line width=1.pt] (2.,1.)-- (2.,-1.);
\draw [line width=1.pt] (4.,-2.)-- (2.,-1.);
\draw [line width=1.pt] (4.,2.)-- (4.,-2.);
\begin{scriptsize}
\draw [fill=black] (0.,1.) circle (2.pt);
\draw (0.,1.3) node {$1$};
\draw [fill=black] (0.,-1.) circle (2.pt);
\draw (0.,-1.3) node {$2$};
\draw [fill=black] (4.,2.) circle (2.pt);
\draw (4.,2.3) node {$5$};
\draw [fill=black] (2.,-1.) circle (2.pt);
\draw (2.,-1.3) node {$4$};
\draw [fill=black] (2.,1.) circle (2.pt);
\draw (2.,1.3) node {$3$};
\draw [fill=black] (4.,-2.) circle (2.pt);
\draw (4.,-2.3) node {$6$};
\draw [fill=black] (7.,-2.) circle (2.pt);
\draw (7.,-2.3) node {$9$};
\draw [fill=black] (6.,0.) circle (2.pt);
\draw (6.3,0) node {$8$};
\draw [fill=black] (7.,2.) circle (2.pt);
\draw (7.,2.3) node {$7$};
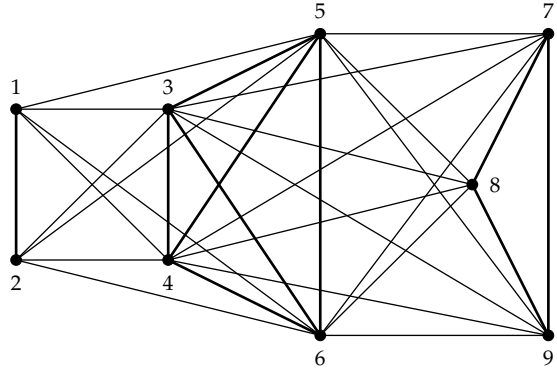
\draw (4.,-3) node {\text{Fig 2.} $G_{4,2} = K_{4}\vee (K_{2}\cup K_{3})$};
\end{scriptsize}
\end{tikzpicture}
\end{center}

Note that $\mathcal{G}(4)=K_{4}\vee (K_{2}\cup K_{3})=G_{4}(K_{2})$, where $G_{4}$ is the graph having a clique of order $7$ and a cluster $(C,S)$, where $\vert C\vert=2$ and $\vert S\vert=4$. Let $\Lambda=\{8,-1^{[3]},-2^{[5]}\}$. From Corollary \ref{main2}, we construct a matrix $M\in \mathcal{S}(\mathcal{G}(4))$ with spectrum $\Lambda$. In fact, let us partition $\Lambda$ as follows:
\begin{align*}
\Lambda=\{8,-1^{[2]}\}\cup \{-1,-2^{[5]}\}=\Lambda_{1}\cup \Lambda_{2}.
\end{align*}
We construct the matrix
\begin{align*}
A=RD_{1}R^{-1}=
\begin{bmatrix}
1 &	1 & 1 \\
1 &	1 &	-1 \\
1 &	-2 & 0
\end{bmatrix}
\begin{bmatrix}
8 &	0 &	0 \\
0 &	-1 & 0 \\
0 &	0 &	-1
\end{bmatrix}
\begin{bmatrix}
\frac{1}{3} & \frac{1}{3} & \frac{1}{3} \\ \\
\frac{1}{6} & \frac{1}{6} &	-\frac{1}{6} \\ \\
\frac{1}{2} & -\frac{1}{2} & 0
\end{bmatrix}=
\begin{bmatrix}
2 &	3 &	3 \\
3 &	2 &	3 \\
3 &	3 &	2
\end{bmatrix}
=(a_{ij}),
\end{align*}
which has spectrum $\Lambda_{1}$. Also, note that $a_{33}=2\notin \Lambda_{2}$. Next, we consider the list $\Gamma=\Lambda_{2}\cup \{2\}$, which has three distinct eigenvalues and choosing $R$ (as in \eqref{I1}) of order $7$. We construct the matrix
\begin{align*}
B=RD_{2}R^{-1}=
\begin{bmatrix}
-\frac{32}{21} & \frac{10}{21} & \frac{10}{21} & -\frac{6}{7} & \frac{1}{7} & \frac{1}{7} & \frac{1}{7} \\ \\
\frac{10}{21} & -\frac{32}{21} & \frac{10}{21} & -\frac{6}{7} & \frac{1}{7} & \frac{1}{7} & \frac{1}{7} \\ \\
\frac{10}{21} & \frac{10}{21} &	-\frac{32}{21} & -\frac{6}{7} & \frac{1}{7} & \frac{1}{7} & \frac{1}{7} \\ \\
-\frac{6}{7} & -\frac{6}{7} & -\frac{6}{7} & \frac{8}{7} & \frac{1}{7} & \frac{1}{7} & \frac{1}{7}
\\ \\
\frac{1}{7} & \frac{1}{7} & \frac{1}{7} & \frac{1}{7} & -\frac{13}{7} & \frac{1}{7} & \frac{1}{7}
\\ \\
\frac{1}{7} & \frac{1}{7} & \frac{1}{7} & \frac{1}{7} & \frac{1}{7} & -\frac{13}{7} & \frac{1}{7} \\ \\
\frac{1}{7} & \frac{1}{7} & \frac{1}{7} & \frac{1}{7} & \frac{1}{7} & \frac{1}{7} & -\frac{13}{7}
\end{bmatrix},
\end{align*}
where $D_{2}=diag\{-1,-2,-2,-2,2,-2,-2\}$. Also, $B{\bf u}=2{\bf u}$, being $${\bf u}^{T}=\frac{1}{2\sqrt{3}}\begin{bmatrix}-1 & -1 & -1 & 3 & 0 & 0 & 0 \end{bmatrix}.$$
Finally, applying Lemma \ref{smi1}, we obtain the matrix
\begin{align*}
M=\begin{bmatrix}
2 &	3 &	-\frac{\sqrt{3}}{2} & -\frac{\sqrt{3}}{2} & -\frac{\sqrt{3}}{2} & \frac{3\sqrt{3}}{2} & 0 & 0 & 0 \\ \\
3 &	2 & -\frac{\sqrt{3}}{2} & -\frac{\sqrt{3}}{2} & -\frac{\sqrt{3}}{2} & \frac{3\sqrt{3}}{2} & 0 & 0 & 0 \\ \\
-\frac{\sqrt{3}}{2} & -\frac{\sqrt{3}}{2} & -\frac{32}{21} & \frac{10}{21} & \frac{10}{21} & -\frac{6}{7} & \frac{1}{7} & \frac{1}{7} & \frac{1}{7} \\ \\
-\frac{\sqrt{3}}{2} & -\frac{\sqrt{3}}{2} & \frac{10}{21} & -\frac{32}{21} & \frac{10}{21} & -\frac{6}{7} & \frac{1}{7} & \frac{1}{7} & \frac{1}{7} \\ \\
-\frac{\sqrt{3}}{2} & -\frac{\sqrt{3}}{2} & \frac{10}{21} &	\frac{10}{21} &	-\frac{32}{21} & -\frac{6}{7} &	\frac{1}{7}	& \frac{1}{7} & \frac{1}{7} \\ \\
\frac{3\sqrt{3}}{2} & \frac{3\sqrt{3}}{2} &	-\frac{6}{7} & -\frac{6}{7} & -\frac{6}{7} & \frac{8}{7} & \frac{1}{7} &	\frac{1}{7}	& \frac{1}{7}
\\ \\
0 & 0 & \frac{1}{7} & \frac{1}{7} & \frac{1}{7} & \frac{1}{7} & -\frac{13}{7} & \frac{1}{7} & \frac{1}{7} \\ \\
0 &	0 & \frac{1}{7}	& \frac{1}{7} &	\frac{1}{7} & \frac{1}{7} &	\frac{1}{7} & -\frac{13}{7} & \frac{1}{7} \\ \\
0 & 0 & \frac{1}{7} & \frac{1}{7} &	\frac{1}{7} & \frac{1}{7} &	\frac{1}{7} & \frac{1}{7} &	-\frac{13}{7}
\end{bmatrix}\in \mathcal{S}(\mathcal{G}(4)),
\end{align*}
which has spectrum $\Lambda$.
\end{example}\ \\

{\bf Acknowledgements.} \\

The research of A. I. Julio was supported by Universidad Cat\'{o}lica del Norte-VRIDT 036-2020, N\'UCLEO UCN VRIDT-083-2020, Chile.\\

{\bf Declaration of Competing Interest}\\

There is no competing interests.

\bigskip
\begin{flushleft}
\textbf{Roberto C. D\'{i}az} \\
\texttt{Departamento de Matem\'aticas}, \\
\texttt{Universidad Cat\'olica del Norte} \\
\texttt{Avenida Angamos 0610}, \\
\texttt{Antofagasta, Chile} \\
\texttt{E-mails}: \textit{rdiaz01@ucn.cl}
\end{flushleft}
\bigskip
\begin{flushleft}
\textbf{Ana I. Julio} \\
\texttt{Departamento de Matem\'aticas}, \\
\texttt{Universidad Cat\'olica del Norte} \\
\texttt{Avenida Angamos 0610}, \\
\texttt{Antofagasta, Chile} \\
\texttt{E-mails}: \textit{ajulio@ucn.cl}
\end{flushleft}

\end{document}